\documentclass[11pt]{article}
\usepackage[colorlinks]{hyperref}
\usepackage{color}
\usepackage{graphicx}
\usepackage{graphics}
\usepackage{makeidx}
\usepackage{showidx}
\usepackage{latexsym}
\usepackage{amssymb}
\usepackage{verbatim}
\usepackage{amsmath}
\usepackage{amsthm}
\usepackage{amsfonts}
\usepackage{subcaption} 
\usepackage[abs]{overpic}
\usepackage{xcolor,varwidth}
\usepackage{geometry}
\newtheorem{theorem}{Theorem}[section]
\newtheorem{proposition}[theorem]{Proposition}
\newtheorem{remark}[theorem]{Remark}
\newtheorem{Claim}[theorem]{Claim}
\newtheorem{question}[theorem]{Question}
\newtheorem{definition}[theorem]{Definition}
\newtheorem{corollary}[theorem]{Corollary}
\newtheorem{lemma}[theorem]{Lemma}
\newtheorem{example}[theorem]{Example}
\newtheorem{conjecture}[theorem]{Conjecture}

\title{Ricci Curvature, Reeb Flows and Contact 3-Manifolds}
\author{Surena Hozoori}
\newcommand{\Addresses}{{
  \bigskip
  \footnotesize

Surena Hozoori, \textsc{Department of Mathematics, Georgia Institute of Technology,
    Atlanta, GA 30332}\par\nopagebreak
  \textit{E-mail address}: \texttt{shozoori3@gatech.edu}

}}
\begin{document}
\maketitle
\noindent
\begin{abstract}
Given a contact 3-manifold we consider the problem of when a given function can be realized as the Ricci curvature of a Reeb vector field for the contact structure. We will use topological tools to show that every admissible function can be realized as such Ricci curvature for a singular metric which is an honest compatible metric away from a measure zero set. However, we will see that resolving such singularities depends on contact topological data and is yet to be fully understood.

\end{abstract}

\section{Introduction}

In Riemannian geometry, it is well known that local restrictions on a Riemannian metric, in particular its curvature tensor, can result in topological consequences. A classical example is the celebrated sphere theorem, introduced by Berger \cite{berger} and Klingenberg \cite{kling} in early 1960s:

\begin{theorem}
Let $(M,g)$ be a Riemannian manifold of arbitrary dimension $n$ with $\frac{1}{4}$-pinched sectional curvature. i.e. if there exists some positive constant $K$, for which $\frac{1}{4}K<Sec(g)\leq K$. Then the universal cover of $M$ is homeomorphic to $\mathbb{S}^n$.
\end{theorem}

In dimension 3, this was generalized extensively by Hamilton and his theory of Ricci flow \cite{hamilton} in 1982:

\begin{theorem}
Let $(M,g)$ be a Riemannian 3-manifold such that $Ricci(g)>0$. Then the universal cover of $M$ is diffeomorphic to $\mathbb{S}^3$.
\end{theorem}
Beside the above {\ttfamily"}rigidity theorems{\ttfamily"}, we also have {\ttfamily"}flexibility theorems{\ttfamily"}, showing the lack of relation to topology. For instance, in 1994 Lokhamp \cite{lokhamp} showed:

\begin{theorem}\label{flex}
Let $M$ be a smooth manifold of arbitrary dimension. Then it admits a Riemannian metric $g$ with $Ricci(g)<0$.
\end{theorem}
\noindent which means negative Ricci curvature does not yield any information about the topology of the underlying manifold.

In this paper, we restrict our attention to dimension 3 and assume that $M$ is an oriented closed connected 3-manifold.

It is natural to ask whether results similar to above theorems hold in other categories of 3-manifolds, since after the proof of geometrization conjecture by Perelman, we can expect to be able to relate topological theories of 3-manifolds to their underlaying Riemannian geometry. On the other hand, we have learned that {\em{\ttfamily"}contact structures{\ttfamily"}}, first showed up in the works of Sophus Lie in 19th century and classically well-studied in different areas of mathematics such as Hamiltonian mechanics and optics, do have subtle and rich relation to the topology of 3-manifolds. Such relation to topology has been discovered since mid 70s and in Bennequin's study of knots in contact manifolds and is now an active area of low dimensional topology, thanks to the development of many topological methods to study contact manifolds, like convex surface theory, open book decompositions, J-holomorphic curves, Heegaard-Floer homology, etc (see \cite{his} for a brief history).

The Riemannian geometry of contact manifolds on the other hand, has been subject of a thorough study in different contexts, by many including Blair, Hamilton, Chern, etc. and by restricting to certain classes of Riemannian metrics, satisfying natural conditions related to the background contact structure (see \cite{bl} for a classical reference). However, we know very little about the global Riemannian geometry of such classes of metrics and therefore their relation to topological aspects of contact structures. A remarkable exception is the analogue of sphere theorem in the category of contact manifolds \cite{EKM1,JH}, when we restrict to a class of Riemannian metrics, namely {\em {\ttfamily"}compatible metrics{\ttfamily"}}, which seem to be a more natural class of metrics from topological point of view (for definitions and discussions related to such class of metrics, see Section~\ref{3}). It is worth mentioning that the class of compatible metrics is just a slight generalization of the well-studied class of {\em {\ttfamily"}contact metrics{\ttfamily"}} \cite{bl}.

\begin{theorem}
Let $(M,\xi)$ be a contact 3-manifold, admitting a compatible metric $g$ with $\frac{1}{4}$-pinched sectional curvature. Then the universal cover of $(M,\xi)$ is contactomorphic to $(\mathbb{S}^3,\xi_{std})$.
\end{theorem}

Note that by Eliashberg's classification of contact structures \cite{elover,elmart}, we have a $\mathbb{Z}$-family of distinct contact structures on $\mathbb{S}^3$. Therefore in the above theorem, the universal cover of $M$ being $\mathbb{S}^3$ is concluded from the classical sphere theorem and specifying the contact structure as {\ttfamily"}the standard contact structure on $\mathbb{S}^3${\ttfamily"} is the consequence of the compatibility condition. A natural generalization would be:

\begin{conjecture}
Let $(M,\xi)$ be a contact 3-manifold, equipped with a compatible metric $g$, such that $Ricci(g)>0$. Then the universal cover of $(M,\xi)$ is contactomorphic to $(\mathbb{S}^3,\xi_{std})$.
\end{conjecture}
\noindent which is currently unknown. 

For more global results, regarding curvature realization of such metrics see \cite{kroug}, about contact topology of compatible metrics with negative {\em{\ttfamily"}$\alpha$-sectional curvatures{\ttfamily"}} see \cite{Hoz}, and regarding the more restricted class of {\em{\ttfamily"}Sasakian metrics{\ttfamily"}}, positive curvature and contact topology in higher dimensions, see \cite{boy}.

Motivated by the above discussion, it is natural to study Ricci curvature realization problems in the category of contact 3-manifolds. In this paper, we study the Ricci curvature of {\em {\ttfamily"}Reeb vector fields{\ttfamily"}} (also known as {\em {\ttfamily"}characteristic vector fields{\ttfamily"}}) associated to a contact manifold. Reeb vector fields have played a central role in contact geometry, going back to its classical development, comparable to Hamiltonian vector fields in symplectic geometry. Moreover, since the early 1990s, we have learned that they can be used to extract contact topological information about the underlying contact manifold as well and by now, we have useful invariants of contact manifold, based on understanding of such dynamics (see \cite{hoferdy} for early developments). Therefore, it is natural to investigate if Ricci curvature of such vector fields contain any contact topological informations and what functions can be realized as such Ricci curvature of a given contact manifolds.

\vskip0.5cm
\begin{question}[Ricci-Reeb Realization Problem]
Given a contact manifold $(M,\xi)$, what functions can be realized as Ricci curvature of the Reeb vector field associated to a compatible metric?
\end{question}

\vskip0.5cm
First we will see that the subtlety of such realization is of global nature, since any function can be realized locally.

\begin{theorem}[Local realization] Let $(M,\xi)$ be a contact 3-manifold equipped with a compatible metric $g$ and $x\in M$ an arbitrary point, and a given function $f:M \to \mathbb{R}$. Then there exists a neighborhood $U$ containing $x$ and a compatible metric $g_*$ such that:

1) $Ricci_* (X_\alpha )(x)=f(x)$ on $U$;

2) $g=g_*$ at  $x$,

\noindent where $Ricci_* (X_\alpha )$ is the Ricci curvature of the Reeb vector field associated with $g_*$.

\end{theorem}

In attempt to extend such solution to a global one, we will use the topological tool of {\em {\ttfamily"}open book decompositions{\ttfamily"}}, which has been widely used in contact topology since the establishment of {\em{\ttfamily"}Giroux's correspondence{\ttfamily"}} between such structures and contact structures in 2000 \cite{giroux}. This method will yield an {\ttfamily"}almost global{\ttfamily"} realization, reducing the pursuit of a global solution to resolving a codimension one embedded submanifold of singularities.

\begin{theorem}[Almost global realization]\label{almostintro}
Let $(M,\xi)$ be a closed oriented contact 3-manifold, $ f(x):M \to \mathbb{R}$ a function on $M$ and $V$ a positive real number. Then there exists a singular metric $g_{\infty}$ and an embedded compact surface with boundary $F \subset M$ such that:

1) $g_\infty$ is a compatible metric on $M \backslash F$;

2) $Ricci(X_\alpha) (x)=f(x)$ on $M \backslash F$, where $X_\alpha$ is the Reeb vector field associated with $g_\infty$;

3) $Vol(g_\infty )=V$;

4) $g_\infty$ can be realized as an element of the completion of the space of compatible Riemannian metrics $\overline{\mathcal{M}_\xi} \subset \overline{\mathcal{M}}$. More precisely, given any $\epsilon >0$,  $[g_\infty]$ is the limit of a $L^2$-Cauchy sequence of compatible metrics $\{ g_n \} \rightarrow [g_\infty]\in \overline{\mathcal{M_\xi}} \subset \overline{\mathcal{M}}\simeq\mathcal{M}_f / \sim$, such that $g_n$ realizes the given function $f(x)$ as $Ricci(X_\alpha)$, outside a $\frac{\epsilon}{2^n}$-neighborhood of $F$.

\end{theorem}

Here, we note that for any compatible metric with {\em {\ttfamily"}instantaneous rotation{\ttfamily"}} $\theta'$ (see Remark~\ref{instan}), we have $Ricci(X_\alpha )\leq \frac{\theta'^2}{2}$ (see Corollary~\ref{curvform}). Therefore in the above theorems, we need to choose the constant $\theta'$ such that $f(x)\leq \frac{\theta'^2}{2}$ (note that $M$ is compact). On the other hand, for a fixed $\theta'$, these theorems hold for any function, respecting such upper bound.

As we will learn about the geometric meaning of such Ricci curvature attaining its maximum (see Proposition~\ref{riccimaxmeaning}), we recognize that the dichotomy of achieving such maximum or not seems to be of central importance for complete understanding of the Ricci-Reeb realization problem. In particular, when considered globally, the dichotomy will result in topological obstructions to realization of a function as $Ricci(X_\alpha)$, showing that the resolution of the singularity set in Theorem~\ref{almostintro} depends on topological data (see Theorem~\ref{topob}).

Using previous works of \cite{AlReg,Ruk,Massot,HT}, we will see that forcing $Ricci(X_\alpha)=\frac{\theta'^2}{2}$ everywhere has strong rigidity consequences for the underlying contact manifold.

  \begin{theorem}
 Let $(M,\xi)$ be a closed contact 3-manifold and $g$ a compatible Riemannian metric with $Ricci(X_\alpha)=\frac{\theta '^2}{2}$ everywhere, where $\theta'$ is the instantaneous rotation of $g$. Then $(M,\xi)$ is finitely covered by Boothby-Wang fibration with $\xi$ being a tight symplectically fillable contact structure. Moreover, if all the periodic Reeb orbits associated with $g$ are non-degenerate, then $(M,\xi)$ is finitely covered by 3-sphere with the standard tight contact structure.

\end{theorem}
  
On the other hand, we can easily find topological obstructions for the extreme opposite case of  nowhere attaining such maximum, i.e. admitting a {\em {\ttfamily"}nowhere Reeb-invariant{\ttfamily"}} compatible metric, strengthening a theorem of Krouglov \cite{kroug}.
 
\begin{theorem}
Let $(M,\xi)$ be any contact 3-manifold with $2e(\xi) \in H^2 (M) \neq 0$. Then for any compatible metric $g$ with instantaneous rotation $\theta'$, there exists some point $x\in M$ at which $Ricci(X_\alpha)(x) = \frac{\theta'^2}{2}$, where $X_\alpha$ is the Reeb vector field associated with $g$.
\end{theorem}

Note that this also means that the analogue of Lokhamp's flexibility theorem, Theorem~\ref{flex}, does not hold in this category.

We will also observe that as long as $(M,\xi)$ admits a compatible metric with $Ricci(X_\alpha)(x) < \frac{\theta'^2}{2}$, we can find a compatible metric for which $Ricci(X_\alpha)$ is arbitrary far from the maximum, confirming the observation that the described dichotomy is of primary importance, compared to other natural dichotomies like $Ricci(X_\alpha)$ being positive versus negative (however, for a survey on the known results concerning the sign of curvature and contact metric geometry see \cite{blsign}):

\begin{theorem} 
Assume $(M,\xi)$ admits some compatible metric with instantaneous rotation $\theta'$ and $Ricci(X_\alpha)<\frac{\theta'^2}{2}$ everywhere. Then for any $c \leq \frac{\theta'^2}{2}$, there exists some compatible metric with instantaneous rotation $\theta'$ and $Ricci(X_\alpha)<c$.

\end{theorem}

It is worth mentioning that we can establish existence of such metric, based on the dynamical assumption of {\em {\ttfamily"}conformal Anosovity{\ttfamily"}} of a contact manifold, i.e. when $(M,\xi)$ admits a conformally Anosov Reeb vector field. Such class of flows were introduced by Eliashberg and Thurston \cite{confoliations} and Mitsumatsu \cite{Mitsumatsu} in mid 1990s and has showed up naturally in the study of Riemannian geometry of contact manifolds by Blair and Perrone \cite{blperr,perr}. We have studied such flows in the category of three dimensional contact topology in \cite{Hoz}.

\begin{theorem}
Let $(M,\xi)$ be a conformally Anosov contact 3-manifold. Then $\xi$ admits a Reeb vector field and a complex structure $J$, satisfying
$$\mathcal{L}_{X_\alpha} J \neq 0$$
everywhere, or equivalently, $(M,\xi)$ admits a compatible metric with instantaneous rotation $\theta'$ and
$$Ricci(X_\alpha)<\frac{\theta'^2}{2}$$
everywhere.
\end{theorem}

However, it is interesting to know whether there are {\ttfamily"}contact topological{\ttfamily"} obstructions to global realization of a given functions, or equivalently resolving the codimension one singularity set described in Theorem~\ref{almostintro}. Based on our study of Ricci-Reeb realization problem and our other result in \cite{Hoz}, we conjecture the following which generalizes the main result of \cite{Hoz}.

 \begin{conjecture}
If $(M,\xi)$ admits a Reeb vector field $X_\alpha$ and a complex structure $J$, satisfying
$$\mathcal{L}_{X_\alpha} J \neq 0$$
everywhere, or equivalently if $(M,\xi)$ admits a compatible metric with instantaneous rotation $\theta'$ and
$$Ricci(X_\alpha)<\frac{\theta'^2}{2}$$
everywhere, then it is tight.
\end{conjecture}

Finally, it is worth mentioning that for our study, we derived a new characterizations of $Ricci(X_\alpha)$ and sectional curvature of planes including $X_\alpha$ ({\em {\ttfamily"}$\alpha$-sectional curvatures{\ttfamily"}}), which we find more natural for topological purposes and more revealing about the interplay of the dynamics of Reeb flows and the underlying compatible geometry (see Remark~\ref{massive}).

\begin{theorem}
Let $(M,\xi)$ be a contact 3-manifold, equipped with a compatible metric $g$, where $\alpha$, $J$, $\theta'$ and $X_\alpha$ are the corresponding contact form, complex structure, instantaneous rotation and the Reeb vector field, respectively. Then for any unit vector $e\in\xi$:
 
 $$k(e,X_\alpha)=g(Je,\nabla_e X_\alpha)^2-g(e,\nabla_e X_\alpha)^2-\frac{\partial}{\partial t}g(e(t),\nabla_{e(t)} X_\alpha)\bigg |_{t=0}$$
 
\noindent where $e(t):=\frac{\tilde{e}(t)}{|\tilde{e}(t)|}$ and $\tilde{e}(t)$ is the unique (locally defined) $\alpha$-Jacobi field with $\tilde{e}(0)=e$. Moreover,
 
 $$Ricci(X_\alpha):= k(e,X_\alpha)+k(Je,X_\alpha)=2G(\xi) \leq \frac{\theta'^2}{2},$$
 where $G(\xi)$ is the {\em extrinsic curvature} of $\xi$.
  \end{theorem}
  
  In Section~\ref{2}, we recall the basic definitions and examples from contact topology and the use of open book decompositions to study contact 3-manifolds. In Section~\ref{3}, primary definitions and properties of compatible Riemannian metrics are given, with emphasis on notions of $\alpha$-Jacobi fields and the second fundamental form of a contact structure. Moreover, we give our characterization of $Ricci(X_\alpha)$ and $\alpha$-sectional curvatures, as well as the proof of local realization theorem. In Section~\ref{4}, we discuss the global aspects of Ricci-Reeb realization problem, by discussing the known global obstructions, recalling elements from metric geometry of the space of Riemannian metrics (mainly based on works of Brian Clark \cite{BC1,BC2}) and proving the almost global realization theorem.
  
  \vskip0.5cm
  \textbf{ACKNOWLEDGEMENT:} I thank my advisor John Etnyre for constant support and helpful discussions and Igor Belegradek for helping me develop a better understanding of related topics from Riemannian geometry. The author was partially supported by the NSF grant DMS-1608684. I dedicate this paper to the memory of Gholamreza H. and Hootan Raeen, a kind father and a beloved friend I lost during this project.
\section{Background From Contact Topology In Dimension 3}\label{2}
In this section, we review some basic notions from contact topology in dimension 3 and the use of open book decompositions in such study. For a more detailed introduction to contact topology, we refer the reader to \cite{Geiges} or \cite{etnyreintro} and for a concise reference on open book decompositions and their role in contact topology, one should consult \cite{OBD}.

\subsection{Contact Structures in Dimension 3}

\begin{definition}\label{contact}
We call the 1-form $\alpha$ a (positive) contact form on $M$, if
$$\alpha \wedge d\alpha > 0,$$
compared to the orientation of $M$. We call $\xi:=\ker{\alpha}$ a (positive) contact structure on $M$. Equivalently, assume $\langle u,v \rangle$ and $\langle u,v,n \rangle$ (locally) form oriented basis for $\xi$ and $TM$, respectively and $g$ is any Riemannian metric on $M$. Then $\xi$ is a (positive) contact structure if $-g([u,v],n)>0$. We call the pair $(M,\xi)$ a contact manifold.
\end{definition}

We can similarly define a negative contact structure. In this paper, we assume contact structures are positive, unless stated otherwise.

\begin{example}

Some classical examples of contact structures include:

1) We call $\xi_{std}:=\ker{\alpha_{std}}$ the standard contact structure on $\mathbb{R}^3$, where $\alpha_{std}=dz-ydx$ is a contact form on $\mathbb{R}^3$.

2) Let $\mathbb{S}^3$ be the unit sphere in $\mathbb{C}^2$. Then it can be seen that $\xi_{std}:=T\mathbb{S}^3 \cap J T\mathbb{S}^3$ (the unique complex line tangent to $\mathbb{S}^3 \subset \mathbb{C}^2$) is a contact structure, which we refer to as the standard contact structure on $\mathbb{S}^3$.

3) $\xi_n:=\ker{\{\cos{2\pi nz}dx-\sin{2\pi nz}dy\}}$ for $n\in \mathbb{N}$ yields an infinite family of contact structures on $\mathbb{T}^3=\frac{\mathbb{R}^3}{\mathbb{Z}^3}$.

4) (Boothby-Wang fibrations) Let $\Sigma$ be a closed oriented surface and $\omega$ an area form on $\Sigma$ with $0\neq [\omega]\in H^2(\Sigma ;\mathbb{Z})$. By Kobayashi \cite{Kob}, there exists an $\mathbb{S}^1$-bundle $\pi:M \to \Sigma$, equipped with the connection form $\alpha$, such that $d\alpha=\pi^* \omega$. It can be easily seen that $\alpha$ is a contact form and we call $(M,\xi:=\ker{\alpha})$ a Boothby-Wang fibration. Introduced by Boothby and Wang  \cite{bw}, these examples can be generalized to higher dimensions by considering any symplectic manifold $(\Sigma^{2n},\omega)$.

\end{example}

By Darboux theorem, contact structures do not have any local invariant, in the sense that for any $p_1 \in (M_1,\xi_1)$ and $q_2\in (M_2,\xi_2)$, there exist open neighborhoods $U_1$ and $U_2$ of $p_1$ and $p_2$, respectively and a diffeomorphism $\phi:U_1\to U_2$, such that $\phi^* (\xi_2)=\xi_1$. i.e. $\phi$ is a {\em {\ttfamily"}contactomorphism{\ttfamily"}}. Therefore the subtly of understanding contact structures is of topological nature and such study has been a prominent topic in low dimensional topology since mid 1970s. One of the main topological properties of contact manifolds, is {\em {\ttfamily"}tightness{\ttfamily"}}, first introduced by Eliashberg \cite{elover}.

 \begin{definition}
 We call $(M,\xi)$ overtwisted if there exists an embedded disk in $M$ that is tangent to $\xi$ along its boundary. Otherwise, we call $(M,\xi)$ tight.
 \end{definition}
 
Distinguishing whether a contact structure is tight or overtwisted is one of central questions is contact topology, since Eliashberg showed \cite{elover,elmart} that the study of overtwisted contact structures can be reduced to algebraic topology of the underlying manifold, while tight contact structures are harder to understand and have more subtle relation to low dimensional topology.

\begin{remark}
It can be shown that all the contact structures given in above examples are tight and the examples given on $\mathbb{S}^3$ and $\mathbb{T}^3$ are the only tight contact structures on those manifolds, while they both admit infinitely many distinct overtwisted contact structures.
\end{remark}
\subsection{Reeb Vector Fields}

It turns out that a certain class of vector fields associated to a contact manifold, namely {\em{\ttfamily"}Reeb vector fields{\ttfamily"}}, gives us a dynamical approach in understanding contact geometry and their relation to topological aspects of these structures is an important part of contact topology since early 1990s.

\begin{definition}
Let $(M,\xi)$ be a contact 3-manifold. Any choice of contact form $\alpha$ for $\xi$ defines a unique vector field $X_\alpha$ satisfying

\center{i) $d\alpha(X_\alpha ,.)=0$;}

\center{ii) $\alpha(X_\alpha)=1$. }
\end{definition}
We can easily observe
\begin{proposition}\label{reebprop}
The Reeb vector field $X_\alpha$ satisfies

a) $X_\alpha \pitchfork \xi$;

b) $\mathcal{L}_{X_\alpha} \alpha =0$ and therefore $\mathcal{L}_{X_\alpha} \xi =0$;

c) On the other hand, any vector field which is transverse to $\xi$ and keeps it invariant is a Reeb vector field for an appropriate choice of contact form.
\end{proposition}
\begin{example}
The Reeb vector fields for the contact structures given above are

1) $\partial_z$ is the Reeb vector field for $(\mathbb{R}^3,\alpha_{std})$.

2) For an appropriate choice of contact form, the Reeb vector field associated to $(\mathbb{S}^3,\xi_{std})$ is tangent to the Hopf fibration on $\mathbb{S}^3$.

3) The vector fields orthonormal to $\xi_n$ (considering the flat metric on $\mathbb{T}^3$) are Reeb vector fields.

4) The integral curves of Reeb vector fields associated to the constructed contact forms on Boothby-Wang fibrations traces the $\mathbb{S}^1$ fibers, described in the construction.
\end{example}

\subsection{Open Book Decompositions and Giroux Correspondence}

Open book decompositions have become one of the main topological tools in contact topology, thanks to the celebrated {\em{\ttfamily"}Giroux correspondence{\ttfamily"}}, established by Emmanuel Giroux in 2000 \cite{giroux}, which was built upon the previous work of Thurston and Wilkelnkemper \cite{thurwilk} and gives a purely topological description of contact structures.

\begin{theorem}[Giroux Correspondence]\label{giroux}
On a given 3-manifold $M$, contact structures up to isotopy are in 1-to-1 correspondence with {\ttfamily"}open book decompositions up to positive stabilization{\ttfamily"}.
\end{theorem}

In this paper, we only use the fact that for any contact structure on a given manifold, there exists an open book decomposition {\em{\ttfamily"}adapted{\ttfamily"}} to it. Therefore, we only include the necessary elements (and exclude describing notions like {\ttfamily"}stabilization of open books{\ttfamily"}).

\begin{definition}
An open book decomposition of a 3-manifold $M$ is a pair $(B,\pi)$ such that $B$ is an oriented link in $M$, referred to as the {\ttfamily"}binding{\ttfamily"} of the open book, and $\pi:M\backslash B \to \mathbb{S}^1$ is a fibration. For any $\tau \in \mathbb{S}^1$, $\pi^{-1}(\tau)$ is the interior of a compact surface $\Sigma_\tau$ with $\partial \Sigma_\tau =B$. We refer to the surfaces $\Sigma_\tau$ as the {\ttfamily"}pages{\ttfamily"} of the open book.
\end{definition}
\begin{example}
1) Considering $\mathbb{S}^3$ as compactified $\mathbb{R}^3$, the $z$-axis can be thought of as the binding of an open book decomposition of $\mathbb{S}^3$, with pages being diffeomorphic to disks.

2) Considering $\mathbb{S}^3 \subset \mathbb{C}^2$ as the unit sphere, the set $B:=\{ (z_1,z_2)\in\mathbb{S}^3 | z_1 z_2 =0 \}$ is the Hopf link and together with the projection $\pi: \mathbb{S}^3 \backslash B \to \mathbb{S}^1:(z_1,z_2) \to \frac{z_1 z_2}{|z_1 z_2 |}$ forms an open book decomposition of $\mathbb{S}^3$.
\end{example}
While Alexander had proved the existence of such structures on any 3-manifold \cite{alexander}, in proof of Theorem~\ref{giroux} Giroux showed that we can construct an open book decomposition {\ttfamily"}adapted{\ttfamily"} to a given contact manifold, in the following sense:

\begin{definition}\label{adapted}
We say the open book decomposition $(B,\pi)$ on $M$ is adapted to the contact structure $\xi$ if there exists some Reeb vector field $X$, such that it is (positively) tangent to $B$ and is (positively) transverse to the pages of $\pi$.
\end{definition}

We note that both open book decompositions in the above example can be isotoped to be adapted to the standard contact structure on $\mathbb{S}^3$.
\section{Local Theory Of Compatibility}\label{3}

In this section, we lay out the background on compatible Riemannian geometry and prove the local realization theorem. In Subsection~\ref{3.1}, we start with basic definitions regarding compatibility and in particular, emphasize on natural geometric notions related to them, like $\alpha$-Jacobi fields and the second fundamental form of a contact structure. In Subsection~\ref{3.2}, we give a new characterization of certain sectional curvatures and Ricci curvature of Reeb vector fields. Finally, in Subsection~\ref{3.3} we will show that by perturbing the complex structure associated to a compatible metric, we can locally realize any function as $Ricci(X_\alpha)$, respecting an upper bound.

\subsection{Compatibility, $\alpha$-Jacobi Fields and Second Fundamental Form of $\xi$}\label{3.1}
On a contact 3-manifold $(M,\xi)$, we can naturally define a Riemannian metric, by choosing a contact form, a complex structure and a positive constant, which measures the {\ttfamily"}rate of rotation{\ttfamily"} of $\xi$.

\begin{definition}
A Riemannian structure $g$ is called {\ttfamily"}compatible{\ttfamily"} with $(M,\xi)$ if 
$$ g(u,v)=\frac{1}{\theta '} d\alpha(u,Jv)+\alpha(u)\alpha(v)$$
for ant $u,v \in TM$, where $\alpha$ is a contact form for $\xi$, $\theta '$ is a positive constant, referred to as {\ttfamily"}instantaneous rotation{\ttfamily"}, and $J$ is a complex structure on $\xi$, naturally extended to $TM$ by first projecting along the Reeb vector field associated with $\alpha$.

\end{definition}
 
 \begin{example}
$(\mathbb{S}^3,\xi_{std})$ and $(\mathbb{T}^3,\xi_n)$ are compatible with round metric on $\mathbb{S}^3$ and flat metric on $\mathbb{T}^3$, respectively.
 \end{example}
 
 \begin{remark}\label{instan}
  1) It can be easily seen that $\theta '=-g([u,v],n)=d\alpha (u,v)$, where $(u,v)$ and $(u,v,n)$ are (locally defined) oriented basis for $\xi$ and $TM$, respectively and therefore, the positivity of $\theta'>0$ is equivalent to the (positive) contact condition in Definition~\ref{contact}. In other words $\theta'$ measures the rate of rotation $\xi$ with respect to {\ttfamily"}being integrable{\ttfamily"}. More precisely, for any point $x\in M$ and basis as above, we can observe that $\theta'=\frac{\partial \theta}{\partial t} \bigg|_{t=0}$, where
  
   $$  \theta (t):= \cos^{-1} \left( \frac{g((\phi_{-t})_* v,n)}{||\phi_{-t})_* v||} \right)$$
 and $\phi_t$ is the flow induced by $u$. We also observe that the area form of $g$ induced on $\xi$ is $\frac{1}{\theta'}d\alpha$ and similarly for the volume form associated with $g$,
  $$Vol(g)=\frac{1}{\theta'}\alpha \wedge d\alpha.$$
  Therefore, such area form and volume form are preserved under $X_\alpha$ by Proposition~\ref{reebprop}.

2) The very well studied class of {\em{\ttfamily"}contact metrics{\ttfamily"}} is the special case of $\theta'=2$ in the above definition (refer to \cite{bl} for the classical literature). However, such restriction is not necessary for our purpose.
 
 \end{remark}
 
 Here, we bring some useful properties of compatible metrics.
 
 \begin{proposition}\label{compreeb}
 For a compatible metric $g$ with associated contact form $\alpha$ and complex structure $J$, we have
 
 1) The Reeb vector field $X_\alpha$ is orthonormal to $\xi$ and moreover, is a geodesic field.
 
 2) The Reeb vector field $X_\alpha$ is divergence free with respect to $g$. Equivalently, for any $e\in\xi$,
 $$g(e,\nabla_e X_\alpha)+g(Je,\nabla_{Je} X_\alpha)=0.$$
 \end{proposition}
 
 \vskip0.75cm
 By Proposition~\ref{compreeb} 1), we have $X_\alpha$ as a geodesic field on $M$ and therefore it is natural to use {\em{\ttfamily"}Jacobi fields{\ttfamily"}} associated to $X_\alpha$, measuring the variations of such geodesic field and therefore helping us understand the dynamics and geometry of Reeb vector fields. 
 
 More precisely, for a point $p \in M$ and $\gamma:[0,\epsilon]\to M$ being a geodesic flow line of $X_\alpha$ with $\gamma(0)=p$, there exists a map
 $$\tilde{\gamma}:[0,\epsilon] \times [0,\epsilon' ] \to M$$
 such that
 
 1) $\frac{\partial \tilde{\gamma}}{\partial t}=X_\alpha$;
 
 2) $\tilde{\gamma}([0,a]\times \{ 0\} )=\gamma$;
 
 3) $v:=\frac{\partial \tilde{\gamma}}{\partial s} |_\gamma$ is orthogonal to $X_\alpha$.
 
 That means that $v$ is a (locally defined) Jacobi field and since for any such map $\tilde{\gamma}$ associated to any geodesic variation, we have $[\frac{\partial \tilde{\gamma}}{\partial t} , \frac{\partial \tilde{\gamma}}{\partial s}]=0$ (see \cite{Sakai} Lemma 2.2), we can characterize (locally defined) $v$ by

$$X_\alpha^2 v(t) +R(v(t),X_\alpha)X_\alpha =0 \text{ (The Jacobi Identity) };$$

 $$X_\alpha v(t)=\nabla_{v(t)}X_\alpha \ ,$$

\noindent where $R$ is the curvature tensor associated to $g$ and forcing the second condition at an initial point suffices. We refer to such $v(t)$ as an {\em{\ttfamily"}$\alpha$-Jacobi field{\ttfamily"}} and note that (locally) $v(t)$ is determined by fixing the initial condition $v(0)$ at $p$ and $v(t)$ is just the push forward of $v(0)$ under $X_\alpha$. We will exploit such vector fields in the proof of Theorem~\ref{curchar}. With the above remark, it is also useful to compute (see \cite{EKM2}):

\begin{proposition}\label{deriv}
For any $e\in \xi$,
$$\nabla_e X_\alpha = J\left(\frac{\theta'}{2} e -\frac{1}{2}(\mathcal{L}_{X_\alpha} J)(e)\right).$$
\end{proposition}

\vskip0.75cm
Now given a Riemannian manifold $(M,g)$, for any oriented plane field $\xi$ with unit normal $n$, we can define {\em{\ttfamily"}the second fundamental form{\ttfamily"}} by:

$$\mathbb{II}(u,v)= g(\nabla_u v,n)$$

\noindent for $u,v \in \xi$.

Notice that such bilinear form is symmetric if and only if $\xi$ is integrable. Nevertheless, we can define two geometric invariants of $\xi$ using this second fundamental form, namely the \textit{mean curvature} $ H(\xi):=trace (\mathbb{II}) $ and the \textit{extrinsic curvature} $ G(\xi):= det (\mathbb{II}(\xi))$. 

By Proposition~\ref{compreeb}, if $(M,\xi)$ is a contact manifold and $g$ a compatible Riemannian metric, we will have:

 $$ H(\xi)= - div_g (X_\alpha)=0,$$
 
 \noindent while we will show in Theorem~\ref{curchar} that $G(\xi)$ can be interpreted as (a constant multiplication of) the \textit{Ricci curvature} of $X_\alpha$.

\subsection{Curvature Characterization}\label{3.2}
\begin{theorem}\label{curchar}
 Let $(M,\xi)$ be a contact 3-manifold, equipped with a compatible metric $g$. Then for any unit vector $e\in\xi$:
 
 $$k(e,X_\alpha)=g(Je,\nabla_e X_\alpha)^2-g(e,\nabla_e X_\alpha)^2-\frac{\partial}{\partial t}g(e(t),\nabla_{e(t)} X_\alpha) \bigg|_{t=0}$$
 where $e(t):=\frac{\tilde{e}(t)}{|\tilde{e}(t)|}$ and $\tilde{e}(t)$ is the unique (locally defined) $\alpha$-Jacobi field with $\tilde{e}(0)=e$. Moreover,
 
 $$Ricci(X_\alpha):= k(e,X_\alpha)+k(Je,X_\alpha)=2G(\xi).$$
  \end{theorem} 
  \begin{proof}
 Since $\nabla_{X_\alpha} X_\alpha =0$ and $ [ X_\alpha , \tilde{e}(t) ]=0$,
$$k(X_\alpha , e)=g(R(e,X_\alpha)X_\alpha , e)=-g(\nabla_{X_\alpha} \nabla_{\tilde{e}} (t) X_\alpha,e)\bigg|_{t=0}$$
$$=-\frac{\partial}{\partial t} g(\nabla_{\tilde{e}(t)} X_\alpha,\tilde{e}(t))\bigg|_{t=0} +g(\nabla_e X_\alpha,\nabla_{X_\alpha} \tilde{e}(t))$$
$$=-\frac{\partial}{\partial t} g(\nabla_{\tilde{e}(t)} X_\alpha,\tilde{e}(t))\bigg|_{t=0} +|\nabla_e X_\alpha|^2$$
$$=-\frac{\partial}{\partial t} \left \{ |\tilde{e} (t) |^2 g(\nabla_{e(t)} X_\alpha,e (t)) \right \} \bigg|_{t=0} +|\nabla_e X_\alpha|^2$$
$$= -2 g(e, \nabla_e X_\alpha)^2 - \frac{\partial}{\partial t} g(\nabla_{e(t)} X_\alpha, e(t))\bigg|_{t=0} +|\nabla_e X_\alpha|^2$$
$$=g(Je,\nabla_e X_\alpha)^2-g(e,\nabla_e X_\alpha)^2-\frac{\partial}{\partial t}g(e(t),\nabla_{e(t)} X_\alpha)\bigg|_{t=0}.$$

Now if we let $e^\perp (t)=\frac{\tilde{e}^\perp (t)}{|\tilde{e}^\perp (t)|}$, where $\tilde{e}^\perp (t)$ is the $\alpha$-Jacobi field with $\tilde{e}^\perp(0)=Je$,

$$Ricci(X_\alpha)= k(e,X_\alpha)+k(Je,X_\alpha)$$
$$=g(Je,\nabla_e X_\alpha)^2 +g(-e,\nabla_{Je} X_\alpha)^2 -g(e,\nabla_e X_\alpha)^2 -g(Je,\nabla_{Je} X_\alpha)^2-...$$

$$...-\frac{\partial}{\partial t} \left \{ g(e(t),\nabla_{e(t)} X_\alpha) +g(e^\perp (t),\nabla_{e^\perp (t)} X_\alpha) \right \}\bigg|_{t=0}$$
$$=\left ( g(Je,\nabla_e X_\alpha) +g(e,\nabla_{Je} X_\alpha) \right ) ^2 -2g(Je,\nabla_e X_\alpha)g(e,\nabla_{Je} X_\alpha)-2g(e,\nabla_e X_\alpha)^2 -...$$
$$...-\frac{\partial}{\partial t} \left \{ g(e(t),\nabla_{e(t)} X_\alpha) +g(e^\perp (t),\nabla_{e^\perp (t)} X_\alpha) \right \}\bigg|_{t=0}$$
$$=2G(\xi)+\left ( g(Je,\nabla_e X_\alpha) +g(e,\nabla_{Je} X_\alpha) \right ) ^2-\frac{\partial}{\partial t} \left \{ g(e(t),\nabla_{e(t)} X_\alpha) +g(e^\perp (t),\nabla_{e^\perp (t)} X_\alpha) \right \}\bigg|_{t=0} .$$

Therefore, the following lemma will complete the proof:
\begin{lemma}
We have:
$$\left ( g(Je,\nabla_e X_\alpha) +g(e,\nabla_{Je} X_\alpha) \right ) ^2 = \frac{\partial}{\partial t} \left \{ g(e(t),\nabla_{e(t)} X_\alpha) +g(e^\perp (t),\nabla_{e^\perp (t)} X_\alpha) \right \}\bigg|_{t=0}.$$
\end{lemma}
\begin{proof}
First compute:
$$g(e(t),\nabla_{e(t)} X_\alpha) +g(e^\perp (t),\nabla_{e^\perp (t)} X_\alpha)=\frac{1}{2} \frac{\partial}{\partial t} \left \{ \ln{|\tilde{e}(t)|^2}+\ln{|\tilde{e}^\perp (t)|^2} \right \}$$
$$=\frac{1}{2} \frac{\partial}{\partial t} \left \{ \ln{|\tilde{e}(t)|^2 |\tilde{e}^\perp (t)|^2} \right \}=-\frac{1}{2} \frac{\partial}{\partial t} \left \{ \ln{\sin^2{\beta (t)}} \right \}=(-\cot{\beta (t)}) \beta' (t)$$
where $\beta (t)$ is the angle between $\tilde{e}(t)$ and $\tilde{e}^\perp (t)$ and we used the fact that Reeb flow preserves the induced area form of $g$ on $\xi$ and therefore, $\tilde{e}(t)\tilde{e}^\perp (t) \sin{ \beta (t) }=1$ for all $t$. Now:
$$\frac{\partial}{\partial t} \left \{ g(e(t),\nabla_{e(t)} X_\alpha) +g(e^\perp (t),\nabla_{e^\perp (t)} X_\alpha) \right \}\bigg|_{t=0}=\left \{ -\cot{\beta (t)}. \beta'' (t) + \csc^2{\beta (t)}.(\beta' (t))^2 \right \}\bigg|_{t=0}$$
$$=(\beta' (0))^2.$$

On the other hand:
$$g(Je,\nabla_e X_\alpha) +g(e,\nabla_{Je} X_\alpha)=g(e+Je,\nabla_{e+Je} X_\alpha)=\frac{\partial}{\partial t}\left \{ \ln{|\tilde{e}(t)+\tilde{e}^\perp (t)|^2} \right \}\bigg|_{t=0}$$
$$=\frac{\partial}{\partial t} \left \{ \ln{ \left ( |\tilde{e}(t)|^2+|\tilde{e}^\perp (t)|^2 +2|\tilde{e} (t)||\tilde{e}^\perp (t)|\cos{\beta (t)} \right ) } \right \}\bigg|_{t=0} $$
$$=\frac{   2|\tilde{e}(t)|\frac{\partial |\tilde{e}(t)|}{\partial t}+2|\tilde{e}^\perp (t)|\frac{\partial |\tilde{e}^\perp (t)|}{\partial t}+2|\tilde{e}(t)|\frac{\partial |\tilde{e}^\perp (t)|}{\partial t}\cos{\beta(t)}+2|\tilde{e}^\perp (t)|\frac{\partial |\tilde{e}(t)|}{\partial t}\cos{\beta(t)}   }{|\tilde{e}(t)+\tilde{e}^\perp (t)|^2} \bigg|_{t=0}-...$$
$$..-\frac{2|\tilde{e}(t)||\tilde{e}^\perp (t)|\sin{\beta (t) }.\beta' (t)    }{|\tilde{e}(t)+\tilde{e}^\perp (t)|^2} \bigg|_{t=0}$$
$$=\left \{ \frac{\partial |\tilde{e}(t)|}{\partial t}+\frac{\partial |\tilde{e}^\perp(t)|}{\partial t} \right \}\bigg|_{t=0}-\beta'(0)=-\beta' (0).$$
which establishes proof of the lemma. In the last equality, we used the fact that
$$0=\frac{\partial}{\partial t} \left \{ \tilde{e}(t)\tilde{e}^\perp (t) \sin{\beta (t)} \right \}\bigg|_{t=0} = \left \{ \frac{\partial |\tilde{e}(t)|}{\partial t}+\frac{\partial |\tilde{e}^\perp(t)|}{\partial t} \right \}\bigg|_{t=0}.$$

\end{proof}
  \end{proof}
  
 Let $( e,Je )$ be any local choice of an orthonormal frame for $\xi$. Using the above characterization, Remark~\ref{instan} and Koszul formula, we can derive the following formula for $Ricci(X_\alpha)$.
  
  \begin{corollary}\label{curvform}
For any $x\in M$, we can write $Ricci(X_\alpha)$ as:
$$Ricci(X_\alpha)(x)=-2P^2(x)+\frac{\theta '^2}{2}-2Q^2(x)$$
where 
$$P(x)=g( e,\nabla_e X)=\frac{1}{\theta '} d\alpha ([e,X],Je)$$
 and 
 $$Q(x)=\frac{\theta '}{2} -g( Je ,\nabla_e X)=\frac{1}{2\theta'} d\alpha ([e,X],e) -\frac{1}{2\theta '} d\alpha ([Je,X],Je) $$
 for any choice of orthonormal frame $( e, Je,X    )$. In particular, 
  $$Ricci(X_\alpha) \leq \frac{\theta '^2}{2}.$$
\end{corollary}

It turns out that $Ricci(X_\alpha)$ attaining its maximum has an important geometric meaning.

\begin{proposition}\label{riccimaxmeaning}
At any point $x\in M$, the followings are equivalent:

(1) $Ricci(X_\alpha)=\frac{\theta'^2}{2};$

(2) $g(e,\nabla_e  X_\alpha)=0$ for any unit vector $e\in \xi$;

(3) $\mathcal{L}_{X_\alpha}J=0;$

(4) $\mathcal{L}_{X_\alpha}g=0.$
\end{proposition}

\begin{proof}
\begin{Claim} At any point $x \in M$, $g(e,\nabla_e X_\alpha)=0$ either for exactly 4 unit vector $e$ at $x$ or for all unit vectors at $x$. 
\end{Claim}
\begin{proof}
Since $g(e,\nabla_e X_\alpha)+g(Je,\nabla_{Je} X_\alpha)=0$ for any $e\in \xi$, there exists some $e\in \xi$ such that $g(e,\nabla_e X_\alpha)=0$. Clearly the same holds for $-e$, $Je$ and $-Je$. Now imagine $g(v,\nabla_{v} X_\alpha)=0$ for some other unit vector $v=ae+bJe$ at $x$ (where $ab\neq 0$). Then

$$0=g(v,\nabla_{v} X_\alpha)=\frac{1}{2} \frac{\partial}{\partial t} \ln{ |\tilde{v}(t)|^2} \bigg|_{t=0}=\frac{1}{2} \frac{\partial}{\partial t} \ln{ |a\tilde{e}(t) +b \tilde{e}^\perp (t)|^2} \bigg|_{t=0}$$
where $\tilde{v}(t)$, $\tilde{e}(t)$ and $\tilde{e}^\perp (t)$ are respectively the $\alpha$-Jacobi field extension of $v$, $e$ and $Je$ respectively. Letting $\beta (t)$ be the angle between $\tilde{e}(t)$ and $\tilde{e}^\perp (t)$, this means
$$0=\frac{1}{2} \frac{\partial}{\partial t} \ln{ \{ a^2 |\tilde{e}(t)|^2 +b^2 |\tilde{e}^\perp (t)|^2 +2ab|\tilde{e} (t)||\tilde{e}^\perp (t) |\cos{\beta(t)}            \}}\bigg|_{t=0} =-ab\beta ' (0)$$
So we have $\beta'(0)=0$. But this computation shows that for any other linear combination $ce+dJe$, we will have $g(ce+dJe,\nabla_{ce+dJe} X_\alpha)=0$, proving the claim.
\end{proof}

$(1)\Rightarrow (2)$ If $Ricci(X_\alpha)=\frac{\theta'^2}{2}$, then $P(x)=g( e,\nabla_e X)=0$ for any choice of unit $e\in \xi $.

$(2)\Rightarrow (1)$ In this case, $g(e+Je, \nabla_{e+Je} X_\alpha)=g(e,\nabla_{Je} X_\alpha)+g(Je,\nabla_e X_\alpha )=0$. Together with Remark~\ref{instan}, this implies
$$-g(e,\nabla_{Je}X_\alpha)=g(Je,\nabla_e X_\alpha)=\frac{\theta'}{2}$$
which implies $P(x)=Q(x)=0$.

$(3)\Rightarrow (2)$ By Proposition~\ref{deriv}, for any $e\in \xi$ we have
$$g(e,\nabla_e X_\alpha )=g(e,J\left(\frac{\theta'}{2} e -\frac{1}{2}(\mathcal{L}_{X_\alpha} J)(e)\right))=g(\frac{\theta'}{2} Je,e)=0.$$

$(1)\Rightarrow (3)$ In this case, for any $e\in \xi$, we have $g(e,\nabla_e X_\alpha)=0$ and $g(Je,\nabla_e X_\alpha )=\frac{\theta'}{2}$. Therefore,
$$g((\mathcal{L}_{X_\alpha} J)(e),e)=g((\mathcal{L}_{X_\alpha} J)(e),Je)=0,$$
which yields $\mathcal{L}_{X_\alpha} J=0$.

$(3)\iff (4)$ The equivalence follows from the fact that in the definition of a compatible metric, $\alpha$ is invariant under $X_\alpha$ and $\theta'$ is constant.

\end{proof}

\begin{remark}\label{massive}
Note that in the above discussion, since $g|_\xi$ is constantly proportional to $d\alpha|_\xi$ and $X_\alpha$ preserves $d\alpha$, after a (local) trivialization of $\xi$, we can (locally) describe the action of such Reeb flow on $\xi$ as a path in $Sp(1_\mathbb{C})$, area preserving linear maps of $\mathbb{R}^2$. Now we can decompose any $A\in Sp(1_\mathbb{C})$ as $A=MU$, where $U\in SO(2)$ measures the rotation of the flow with respect to the trivialization and $M$ is a positive definite matrix measuring the hyperbolicity of $A$. On the other hand, using the above notation, we have $g(e,\nabla_e X_\alpha)=\frac{1}{2} \frac{\partial}{\partial t} \ln{|\tilde{e}(t)|^2}$ measuring the infinitesimal rate of change of the length of vectors in $\xi$ with respect to $g$. Therefore by Theorem~\ref{curchar}, the deviation of $Ricci(X_\alpha)$ from its maximum measures the {\em{\ttfamily"}infinitesimal hyperbolicity{\ttfamily"}} of the flow of $X_\alpha$ with respect to $g$, leaving the rotation of the flow undetected. In other words, when $Ricci(X_\alpha)=\frac{\theta'^2}{2}$ at a point, the flows acts as pure rotation infinitesimally, while when $Ricci(X_\alpha)<\frac{\theta'^2}{2}$ in a neighborhood, we (locally) have a section $e$ of $\xi$ such that $g(e,\nabla_e X_\alpha)=g(Je,\nabla_{Je} X_\alpha)=0$ with $g(.,\nabla_. X_\alpha)$ having alternating signs in the intermediate regions (see Proposition~\ref{compreeb}, Proposition~\ref{riccimaxmeaning}). See Figure 1.

 Although in Theorem~\ref{localreal}, we will see that by manipulation of $g$, we can hide such hyperbolicity locally, the global consequences of such dynamical phenomena can be of (contact) topological interest (see the discussion in Subsection~\ref{topob}).

We note that the case when we have such rigidity everywhere, is studied previously either as {\em{\ttfamily"}K-contact structures{\ttfamily"}} or {\em{\ttfamily"}geodesible contact structures{\ttfamily"}} \cite{Massot}, defined to be contact manifolds equipped with compatible metrics satisfying $\mathcal{L}_{X_\alpha} J=0$ everywhere and compatible metrics whose geodesics tangent to $\xi$ at a point remains tangent to $\xi$ (which is equivalent to $g(e,\nabla_e X_\alpha )=0$ for any $e\in \xi$), respectively.

On the other hand, the term $\frac{\partial}{\partial t}g(e(t),\nabla_{e(t)} X_\alpha) |_{t=0}$ in the computation of $k(e,X_\alpha)$ can help us measure the infinitesimal rotation of $X_\alpha$ with respect to splitting of $TM$ described above, in the case of $Ricci(X_\alpha)<\frac{\theta'^2}{2}$. When considered globally, this viewpoint can potentially help us measure the rotation of Reeb fields with respect to a trivialization and hence, achieve topological information. See \cite{Hoz} for a use of such viewpoint.

\end{remark}
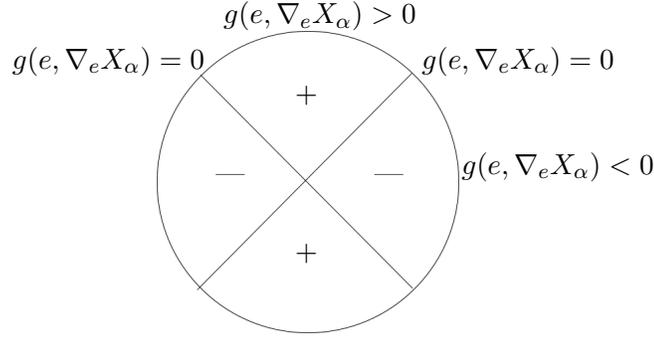
\begin{figure}
  \center \begin{overpic}[width=4cm]{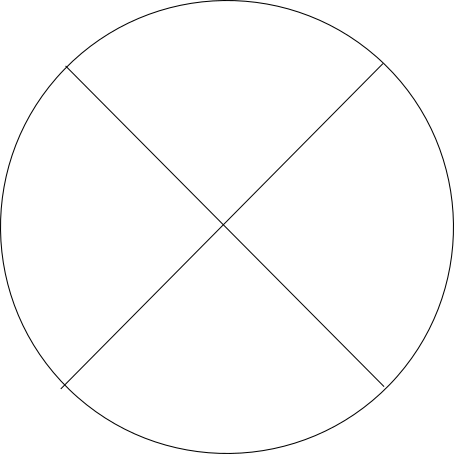}
  \put(-55,100){$g(e,\nabla_e X_\alpha)=0$}
  \put(100,100){$g(e,\nabla_e X_\alpha)=0$}
    \put(25,117){$g(e,\nabla_e X_\alpha) >0$}
      \put(115,60){$g(e,\nabla_e X_\alpha)<0$}
  \put(52,87){+}
  \put(22,57){\textemdash}
    \put(52,27){+}
  \put(82,57){\textemdash}
  
  \end{overpic}
\caption{Splitting of $\xi$ when $\mathcal{L}_{X_\alpha}g\neq 0$ and regions with alternating signs for $g(e,\nabla_e X_\alpha)$}
\end{figure}

\begin{corollary}\label{cormax}
The followings are equivalent.

1) $Ricci(X_\alpha)=\frac{\theta'^2}{2}$ everywhere;

2) $\xi$ is K-contact;

3) $\xi$ is geodesible.

\end{corollary}

\subsection{Deformation and Local Realization}\label{3.3}

In order to prescribe a function for $Ricci(X_\alpha)$, we want to understand the effect of perturbing $J$ on $Ricci(X_\alpha)$.

\begin{lemma}[Perturbation of complex structure]\label{pert}
Let $(M,\xi)$ be equipped with a compatible metric $g(.,.)=\frac{1}{\theta '}d\alpha (.,J.)+\alpha(.)\alpha(.)$ and assume that there exist a line sub bundle of $\xi$ (equivalently $2e(\xi)=0$). Define a new complex structure by 
$$ J_*: \ e \mapsto \eta ^2 Je +\lambda e$$
where $e$ is any vector on the above line section, $\lambda$ is any function on $M$ and $\eta$ is a positive function on $M$. The Ricci curvature for the new compatible metric $g_* (.,.)=\frac{1}{\theta '}d\alpha (.,J_*.)+\alpha(.)\alpha(.)$ is given by
$$Ricci_*(X_\alpha)(x)= -2\left ( P_*(x) + \frac{X_\alpha\eta}{\eta} \right )^2 +\frac{\theta'^2}{2} - 2\left ( Q_* (x) - \frac{\lambda}{2 \eta} X_\alpha\eta +\frac{\eta}{2} X_\alpha(\frac{\lambda}{\eta}) \right )^2$$
where
$$P_* (x)= \frac{1}{\theta'} d\alpha ([e,X_\alpha ],Je) +\frac{1}{\theta'} \frac{\lambda} {\eta ^2} d\alpha ([e,X_\alpha ],e)$$
and
$$Q_* (x)= \frac{1}{2\theta'}  \frac{1}{\eta ^2} d\alpha ([e,X_\alpha],e) -\frac{\eta ^2}{2\theta'}d\alpha ([Je,X_\alpha],Je)  -\frac{\lambda}{\theta'} d\alpha ([e,X_\alpha ],Je) -\dots$$
$$\dots -\frac{1}{2\theta'}\frac{\lambda^2}{\eta^2}d\alpha([e,X_\alpha ],e).$$
\end{lemma}

\begin{proof}
Let $\overset{*}{\nabla}$ be the Levi-Civita connection associated to $g_*$. Note that under the above perturbation the length of $e$ will become $\eta$. So $\frac{e}{\eta}$ will be the unit vector in the direction of $e$. Applying Koszul formula as in Corollary~\ref{curvform} we have

$$g_* (\frac{e}{\eta} ,\overset{*}{\nabla} _{\frac{e}{\eta}} X_\alpha)=
\frac{1}{\theta'} d\alpha ([\frac{e}{\eta},X_\alpha ],J_* \frac{e}{\eta})=
\frac{1}{\theta'} d\alpha (\frac{1}{\eta}[e,X_\alpha ]-X_\alpha (\frac{1}{\eta})e,\eta Je+\frac{\lambda}{\eta}e)$$
$$=\frac{1}{\theta'} d\alpha ([e,X_\alpha],Je) +\frac{1}{\theta'} \frac{\lambda} {\eta ^2} d\alpha ([e,X_\alpha],e)+ \frac{X_\alpha \eta}{\eta}$$
We will also have
$$g_* ( J_* \frac{e}{\eta} ,\overset{*}{\nabla} _{\frac{e}{\eta}} X_\alpha)=\frac{\theta'}{2} - \frac{1}{2\theta'} d\alpha ([\frac{e}{\eta},X_\alpha ],\frac{e}{\eta}) +\frac{1}{2\theta '} d\alpha ([J_* \frac{e}{\eta},X_\alpha ],J_* \frac{e}{\eta})$$
$$=\frac{\theta'}{2} - \frac{1}{2\theta'} d\alpha (\frac{1}{\eta}[e,X_\alpha]+\frac{X_\alpha\eta}{\eta^2} e,\frac{e}{\eta}) +\frac{1}{2\theta '} d\alpha (\eta [Je,X_\alpha]-(X_\alpha\eta)Je+\frac{\lambda}{\eta}[e,X_\alpha]-(X_\alpha\frac{\lambda}{\eta})e,\eta Je+\frac{\lambda}{\eta} e)$$
$$= \frac{\theta '}{2} - \frac{1}{2\theta'}  \frac{1}{\eta ^2} d\alpha ([e,X_\alpha],e) +\frac{\eta ^2}{2\theta'}d\alpha ([Je,X_\alpha],Je) +\frac{\lambda}{2\theta'}d\alpha([Je,X_\alpha],e) +\frac{\lambda}{2\theta'} d\alpha ([e,X_\alpha],Je) +\dots$$
$$\dots +\frac{1}{2\theta'}\frac{\lambda^2}{\eta^2}d\alpha([e,X_\alpha],e) +\frac{\lambda}{2\eta} X_\alpha\eta -\frac{\eta}{2} X_\alpha\frac{\lambda}{\eta}.$$

\end{proof}

As a result, starting from any compatible metric, it is enough to perturb the associated complex structure to realize any function as $Ricci(X_\alpha)$ locally, assuming it respects the upper bound on Ricci curvature.

\begin{theorem}[Local realization]\label{localreal}
 Let $(M,\xi)$ be a contact 3-manifold equipped with a compatible metric $g$ and $x\in M$ an arbitrary point, and $f:M \rightarrow \mathbb{R}$ a function such that $f(x) \leq \frac{\theta'^2}{2}$. Then there exists the neighborhood $U$ containing $x$ and a compatible metric $g_*$ with instantaneous rotation $\theta'$ such that

1) $Ricci_* (X_\alpha )(x)=f(x)$ on $U$;

2) $g=g_*$ at  $x$.

\end{theorem}

\begin{proof}
Let $\mu =\frac{\lambda}{\eta ^2}$. After choosing local trivialization $e$, we can rewrite the equations of Lemma~\ref{pert} for the corresponding perturbation of the almost complex structure $J$:

$$Ricci_*(X_\alpha) (x)=-2\left ( \hat{P_*}(x)+X_\alpha (\ln{\eta}) \right )^2 +\frac{\theta'^2}{2} - 2\eta^4 \left(\hat{Q_*}(x)+\frac{1}{2\theta'} \frac{1}{\eta^4}d\alpha ([e,X_\alpha ],e)+\frac{1}{2} X_\alpha \mu \right)^2$$
where
$$\hat{P_*}(x)=\frac{1}{\theta'} d\alpha ([e,X_\alpha ],Je) + \frac{1}{\theta'} \mu d\alpha ([e,X_\alpha ],e)$$
and 
$$\hat{Q_*}(x)= -\frac{1}{2\theta'}d\alpha ([Je,X_\alpha],Je)  -\frac{\mu}{\theta'} d\alpha ([e,X_\alpha],Je)  -\frac{1}{2\theta'} \mu^2 d\alpha([e,X_\alpha],e)$$

Now in order to solve the PDE $Ricci_* (X_\alpha)(x)=f(x)$ locally, let $U$ be an open neighborhood around $x$ such that $x \in \Sigma_0 \subset U$, where $\Sigma_0$ is a (local) smooth surface transverse to $X_\alpha$ including $x$ and $X_\alpha$ gives the neighborhood $U$ a smooth product structure $U \simeq \Sigma_0 \times (-\epsilon ,\epsilon )$. Now, we can solve our PDE on $U$, by solving the following two PDEs.
$$(1) \begin{cases}
\hat{P_*}(x)-X_\alpha(\ln{\eta})=0 \\
\eta |_{\Sigma_0}=1
\end{cases}$$
$$(2) \begin{cases}
\frac{\theta'^2}{4} -\eta^4 (\hat{Q_*}(x)+\frac{1}{2\theta'} \frac{1}{\eta^4}d\alpha ([e,X_\alpha],e)+\frac{1}{2} X_\alpha \mu)^2= \frac{f(x)}{2} \\
\mu |_{\Sigma_0}=0
\end{cases}$$
But exploiting the (local) product structure above, we can translate these two PDEs into two ODEs on $\Sigma_0$.
$$(1) \begin{cases}
\frac{\partial}{\partial t} \ln{\eta} = -\frac{1}{\theta'} d\alpha ([e,X_\alpha],Je) -\frac{1}{\theta'} \mu d\alpha ([e,X_\alpha],e)\\
\eta (0)=1
\end{cases}$$
$$(2) \begin{cases}
\frac{1}{2} \frac{\partial}{\partial t} \mu = \frac{1}{\eta^2} \sqrt{\frac{\theta'^2}{4}-\frac{f(x(t))}{2}} -\frac{1}{2\theta'} \frac{1}{\eta^4}d\alpha ([e,X_\alpha],e) +\frac{1}{2\theta'}d\alpha ([Je,X_\alpha],Je)+...\\
\hskip6cm...  +\frac{\mu}{\theta'} d\alpha ([e,X_\alpha],Je)  +\frac{1}{2\theta'} \mu^2 d\alpha([e,X_\alpha],e)\\
\mu (0)=0
\end{cases}$$
Now because of existence and uniqueness of the solution of ODEs, we can solve these two equations in the following way. First solve (1) for $\eta$ in terms of $\mu$. More explicitly,
$$\eta (x(t))= e^{\int_0^t \hat{P_*}(x(s))ds}$$
which depends on the unknown $\mu(x(t))$ (note that $\eta$ stays positive). But replacing this solution (in terms of $\mu$) into (2), we will have another ODE
$$(2) \begin{cases}
\frac{\partial}{\partial t} \mu (x(t))=F(x(t),\mu) \\
\mu (0)=0
\end{cases}$$
for the appropriate function $F$. Now we can locally solve this ODE to find $\mu$. Replacing this into the solution for $\eta$ which was in terms of $\mu$, we find $\eta$. Hence, we also have found $\lambda=\mu \eta^2$. The complex structure defined by these two parameters will define the desired Riemannian metric $g_*$. Notice that $g=g_* |_{\Sigma_0}$ by our initial conditions.
\end{proof}

\section{Open Book Decompositions and Almost Global Realization}\label{4}

The goal of this section is to establish the {\ttfamily"}almost global realization theorem{\ttfamily"}. In Subsection~\ref{topob}, we discuss how such realization can depend on the underlying (contact) topological information. In Subsection~\ref{4.2}, we recall elements from metric geometry of the space of (compatible) Riemannian metrics, mostly due to Brian Clarke \cite{BC1,BC2}. In Subsection~\ref{4.3}, we give the proof for the almost realization theorem.

\subsection{Topological Obstructions}\label{topob}

First we note that forcing $Ricci(X_\alpha)$ to obtain its maximum everywhere restricts the contact topology significantly, since this is equivalent to $\mathcal{L}_{X_\alpha} J=0$ everywhere. Putting the previous works of previous works of \cite{AlReg,Ruk,Massot,HT} together, we have

\begin{theorem}
 Let $(M,\xi)$ be a closed contact 3-manifold and $g$ a compatible Riemannian metric with instantaneous rotation $\theta'$, such that $Ricci(X_\alpha)=\frac{\theta '^2}{2}$ everywhere. Then $(M,\xi)$ is finitely covered by a Boothby-Wang fibration with $\xi$ being a tight symplectically fillable contact structure. Moreover, if all the periodic Reeb orbits associated with $g$ are {\ttfamily"}non-degenerate{\ttfamily"}, i.e. their Poincare return map does not have 1 as eigenvlaue, then $(M,\xi)$ is finitely covered by $(\mathbb{S}^3,\xi_{std})$.

\end{theorem}
\begin{proof}

The implication follows from classification of K-contact structures by Rukimbira \cite{Ruk} (see Corollary~\ref{cormax}). In fact, after an arbitrary small perturbation $(M,\xi,g)$ can be approximated by arbitrary close {\ttfamily"}almost regular{\ttfamily"} K-contact structure. i.e. $X_\alpha$ induces a $S^1$ action as Killing vector field. It turns out \cite{AlReg} that this induces Seifert fibration structure on $(M,\xi)$ whose fibers keep $\xi$ invariant. This is called a {\ttfamily"}generalized Boothby-Wang fibration{\ttfamily"} and is finitely covered by a Boothby-Wang fibration. Furthermore, \cite{Nied} shows that these contact structures are symplectically fillable and tight.

Moreover, since in this case $X_\alpha$ preserves the length of any vector $e\in \xi$, all periodic orbits will be {\ttfamily"}elliptic{\ttfamily"}. i.e. have (complex) Poincare return map with unit length eigenvalues. If furthermore, all periodic orbits are non-degenerate as well, \cite{HT} shows that $(M,\xi)$ is either $(\mathbb{S}^3,\xi_{std})$ or a Lens space with $\xi$ being universally tight.
\end{proof}

Now it is also interesting to understand the extreme opposite of the above situation. That is when $(M,\xi)$ admits a {\em{\ttfamily"}nowhere Reeb-Invariant{\ttfamily"}} compatible metric. i.e. a metric for which $\mathcal{L}_{X_\alpha} g \neq 0$ everywhere. First, we easily observe that there are algebraic obstructions for the existence of such metrics, improving \cite{kroug}.

\begin{theorem}
Let $(M,\xi)$ be any contact 3-manifold with $2e(\xi) \in H^2 (M) \neq 0$. Then for any compatible metric $g$ with instantaneous rotation $\theta'$, there exists some point $x\in M$ at which $Ricci(X_\alpha)(x) = \frac{\theta'^2}{2}$, where $X_\alpha$ is the Reeb vector field corresponding to $g$.
\end{theorem}

\begin{proof}
The proof immediately follows from the fact that if we have$Ricci(X_\alpha)<\frac{\theta'}{2}$ everywhere (see Remark~\ref{massive}), there exists a (unique up to homotopy) line field $\langle e \rangle \subset \xi$ with $g(e,\nabla_e X_\alpha)>0$, and therefore $\xi$ admits a globally defined line field. By \cite{Kob}, this is equivalent to $2e(\xi) \in H^2 (M) = 0$.
\end{proof}

However, we still do not know whether this is the only obstruction or if there are others of contact topological nature. In fact, in \cite{Hoz} we 
conjectured the following statement in support of the latter viewpoint, which can be seen to partly generalize the main theorem of \cite{Hoz} about {\em{\ttfamily"}conformally Anosov{\ttfamily"}} contact 3-manifolds.

\begin{conjecture}
If $\xi$ admits a Reeb vector field and a complex structure $J$, satisfying
$$\mathcal{L}_{X_\alpha} J \neq 0$$
everywhere, or equivalently if $(M,\xi)$ admits a compatible metric with instantaneous rotation $\theta'$ and
$$Ricci(X_\alpha)<\frac{\theta'^2}{2}$$
everywhere, then it is tight.
\end{conjecture}

It is worth mentioning that using our computation, we can see that when it does admit such compatible metric, then we can make $Ricci(X_\alpha)$ arbitrary far from the upper bound, confirming the significance of the dichotomy discussed above.

\begin{theorem} 
Assume $(M,\xi)$ admits some compatible metric with instantaneous rotation $\theta'$ and $Ricci(X_\alpha)<\frac{\theta'^2}{2}$ everywhere, in particular if $(M,\xi)$ is conformally Anosov. Then for any $c \leq \frac{\theta'^2}{2}$, there exists some compatible metric with instantaneous rotation $\theta'$ and $Ricci(X_\alpha)<c$.

\end{theorem}
\begin{proof}
Since $(M,\xi)$ admits a metric with $Ricci(X_\alpha)<\frac{\theta'^2}{2}$, we have $2e(\xi)=0\in H^2(M)$.  Then there exists a line sub bundle $\langle e \rangle \subset \xi$. Choose some contact from $\alpha$ and complex structure $J$. For some constant $\lambda$ define a perturbation of complex structure $J_\lambda :\langle e \rangle  \to J\langle e \rangle +\lambda \langle e \rangle $. Letting $\eta=1$ and $X\lambda=0$ in Lemma~\ref{pert}, we have
$$Ricci_\lambda(X_\alpha) (x)= -2\left(P_\lambda(x)\right)^2 +\frac{\theta'^2}{2} - 2\left(Q_\lambda (x) \right)^2$$
where
$$P_\lambda (x)= -\frac{1}{\theta'} d\alpha ([e,X],Je) -\frac{1}{\theta'} \lambda d\alpha ([e,X],e)$$
and
$$Q_\lambda (x)= \frac{1}{2\theta'}   d\alpha ([e,X],e) -\frac{1}{2\theta'}d\alpha ([Je,X],Je)  -\frac{\lambda}{\theta'} d\alpha ([e,X],Je) -\dots$$
$$\dots -\frac{1}{2\theta'} \lambda^2d\alpha([e,X],e)$$
So $Ricci_\lambda(X_\alpha) (x)$ is a non-constant (since we start with $Ricci(X_\alpha)<\frac{\theta'^2}{2}$) polynomial with even degree in terms of $\lambda$ and function coefficients. At each point, we can choose $\lambda$ such that we have $Ricci_\lambda(X_\alpha) (x)<c$ at that point. Since $M$ is compact, we can choose such $\lambda$ globally.
\end{proof}

Finally, we note that the existence of a nowhere-Reeb invariant metric can be concluded, under the dynamical assumption of conformal Anosovity on $(M,\xi)$. A conformal Anosov contact manifold is a contact manifolds $(M,\xi)$ admitting a conformally Anosov Reeb vector field. i.e. some $X_\alpha$ and the continuous $X_\alpha$-invariant splitting $\xi \simeq E^s \oplus E^u$, such that for any $u\in E^s$ and $v \in E^u$,
$$ ||\phi_*^t (v)  || / ||\phi_*^t (u) || \geq Ae^{Ct}||v || /||u ||;$$
where $\phi^t$ is the flow of $X_\alpha$ and $A,C>0$ are positive constants.

It is easy to see \cite{confoliations,Mitsumatsu} that conformal Anosovity of $X_\alpha$ is equivalent to $\langle X_\alpha \rangle =\xi_+ \cap \xi_-$, where $\xi_+$ and $\xi_-$ are transverse positive and negative contact structures on $M$. Now if we (locally) consider sections $e_+ \in \xi \cap \xi_+$ and $e_-\in\xi\cap\xi_-$ such that $(e_+ , e_- )$ form an oriented basis for $\xi$, positivity of $\xi_+$ and negatively of $\xi_-$ will imply $g([e_+,X_\alpha],e_-)>0$ and $g([e_-,X_\alpha],e_+)>0$, respectively. Therefore, the dynamics of $X_\alpha$ cannot be {\em{\ttfamily"}purely rotational{\ttfamily"}} (see Figure 2) and by discussion in Remark~\ref{massive}, we have

 \begin{figure}\label{fig2}

  \center \begin{overpic}[width=4cm]{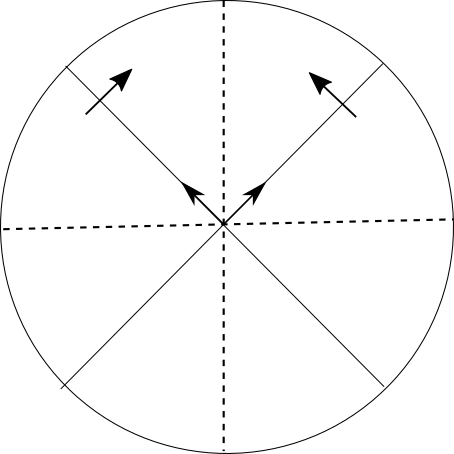}
  \put(72,65){$e_+$}
  \put(32,65){$e_-$}
         \put(10,105){$\xi_-$}
         \put(95,105){$\xi_+$}
         \put(57,120){$E^u$}
         \put(120,60){$E^s$}
       \end{overpic}
    \caption{Confomally Anosov dynamics}
    \label{fig:1}

  \hspace{2cm}

\end{figure}
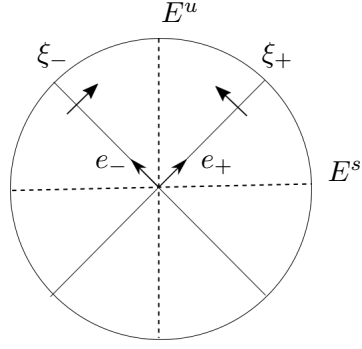
\begin{theorem}
Let $(M,\xi)$ be a conformally Anosov contact 3-manifold. Then $\xi$ admits a Reeb vector field and a complex structure $J$, satisfying
$$\mathcal{L}_{X_\alpha} J \neq 0$$
everywhere, or equivalently $(M,\xi)$ admits a compatible metric with
$$Ricci(X_\alpha)<\frac{\theta'^2}{2}$$
everywhere.
\end{theorem}

\subsection{Completion Of The Space of Compatible Metrics}\label{4.2}

In Subsection~\ref{topob}, we observed the contact topological subtlety of finding the global solutions to Ricci-Reeb realization problem and we can ask what is the best we can do to realize a function as $Ricci(X_\alpha)$. In order to establish almost global solutions to the Ricci-Reeb realization problem, we need some elements from the geometry of the space of Riemannian metrics on $M$, denoted by $\mathcal{M}$. Although the Riemannian geometry of $\mathcal{M}$, like geodesics, sectional curvature, etc. is studied in the classical literature, its metric geometry and in particular, its completion, was not understood well, until relatively recently, in the works of Brian Clarke \cite{BC1,BC2}.

It can be seen that $\mathcal{M}$ admits a natural Riemannian metric, often called $L^2$-metric, denoted by $(.,.)$ and induced from its inclusion into $S^2 T^* M$, the space of symmetric (0,2)-tensor fields on $M$. Let $g \in \mathcal{M}$ and $h,k \in T_g \mathcal{M}$:

$$(h,k):= \int_M trace(g^{-1} h g^{-1} k) dVol(g).$$

Notice that this is generalization of Weil-Peterson metric in Teichmuller theory. This inner product naturally defines a distance function $d$ on $\mathcal{M}$, which satisfies the following interesting and useful property, letting us control the distance between two metrics by controlling the volume of the set they differ on.

\begin{proposition}\label{cool}
Let $g_0 , g_1 \in \mathcal{M}$ and $E:=\{ x \in M | g_0 (x)=g_1 (x)      \}$. Then
$$d(g_0 , g_1 ) \leq C \left ( \sqrt{Vol(E,g_1)} + \sqrt{Vol(E,g_0)}   \right ),$$
where $C$ is a constant only depending on the dimension of $M$ and $Vol(E,g_i)$ is the volume of $E$ measured by $g_i$ for $i\in\{ 0,1 \}$.
\end{proposition}

Brian Clark characterized the completion of $\mathcal{M}$ as follows. Let $\overline{\mathcal{M}}$ be such completion and $\mathcal{M}_f$ be the space of measurable, symmetric, finite volume semi-metrics on $M$.

\begin{theorem}\label{metchar}
Using the above notations, we have the natural identification
$$\overline{\mathcal{M}}\simeq\mathcal{M}_f / \sim ,$$
where for $g_0 , g_1\in \mathcal{M}_f$, we have $g_0 \sim g_1$ if and only if for almost any $x \in M$, $g_0 (x)=g_1 (x)$ when at least one of them is non-degenerate. Such identification can be improved to an isometry.
\end{theorem}

Moreover, in order to understand $L^2$-limit of metrics, we need to control how metrics {\ttfamily"}degenerate{\ttfamily"} on measurable subsets of $M$.

\begin{definition}
Let $\tilde{g}\in \mathcal{M}_f$. We define
$$X_{\tilde{g}}:=\{ x\in M | \tilde{g} (x) \text{ is degenerate}    \} \subset M,$$
which we call the {\ttfamily"}deflated{\ttfamily"} set of $\tilde{g}$.
\end{definition}

\begin{definition}
Let $\{ g_k \}_{k\in\mathbb{N}} \subset M$ be any sequence. We define the set
$$D_{\{ g_k  \}_{k\in\mathbb{N}}} :=\{ x \in M | \forall \delta >0, \exists k \in \mathbb{N} \text{ s.t. } \det{G_k (x)} <\delta         \},$$
where $G_k$ is $g$-dual of $g_k$ for some fixed $g\in \mathcal{M}$.We call $D_{\{ g_k  \}_{k\in\mathbb{N}}}$ the {\ttfamily"}deflated{\ttfamily"} set of $\{   g_k \}$. This definition does not depend on the choice of $g$.

\end{definition}

Although the conditions of convergence in the following theorem can be relaxed extensively, in order to avoid introducing further notions, we give the following theorem which suffices for our purpose (see \cite{BC1} Definition 4.4, Theorem 4.3 and Theorem 5.19).

\begin{theorem}\label{convchar}
Using the above characterization of $\overline{\mathcal{M}}$, we have
$$ \{g_k \} \to [ g_\infty ],$$
if $ \{g_k \}$ is $d$-Cauchy and

1) $\Sigma_{k=1}^\infty d(g_k , g_{k+1})<\infty$;

2) $X_{g_\infty}$ and $D_{ \{ g_k \} }$ differ at most by a nullset;

3) $g_k(x) \to g_\infty (x)$ for almost every $x\in M \backslash D_{\{ g_k\}}$.

\end{theorem}


\subsection{The Proof of Almost Global Realization Theorem}\label{4.3}

\begin{theorem}
Let $(M,\xi)$ be a closed oriented contact 3-manifold, $\frac{\theta'^2}{2} \geq f(x):M \to \mathbb{R}$ a function on $M$ and $V$ a positive real number. Then there exists a singular metric $g_{\infty}$ with instantaneous rotation $\theta'$ and an embedded compact surface with boundary $F \subset M$ such that

1) $g_\infty$ is a compatible metric on $M \backslash F$,

2) $Ricci(X_\alpha) (x)=f(x)$ on $M \backslash F$, where $X_\alpha$ is the Reeb vector field associated with $g_\infty$,

3) $Vol(g_\infty )=V$,

4) $g_\infty$ can be realized as an element of the completion of the space of compatible Riemannian metrics $\overline{\mathcal{M}_\xi} \subset \overline{\mathcal{M}}$. More precisely, given any $\epsilon >0$,  $[g_\infty]$ is the limit of a $L^2$-Cauchy sequence of compatible metrics $\{ g_n \} \rightarrow [g_\infty]\in \overline{\mathcal{M_\xi}} \subset \overline{\mathcal{M}}\simeq\mathcal{M}_f / \sim$, such that $g_n$ realizes the given function as $Ricci(X_\alpha)$, outside a $\frac{\epsilon}{2^n}$-neighborhood of $F$.

\end{theorem}
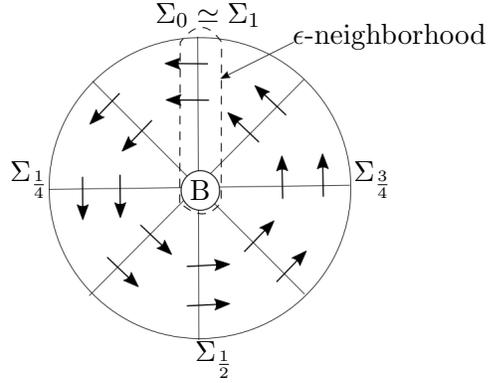
\begin{figure}\label{fig3}
  \center \begin{overpic}[width=4cm]{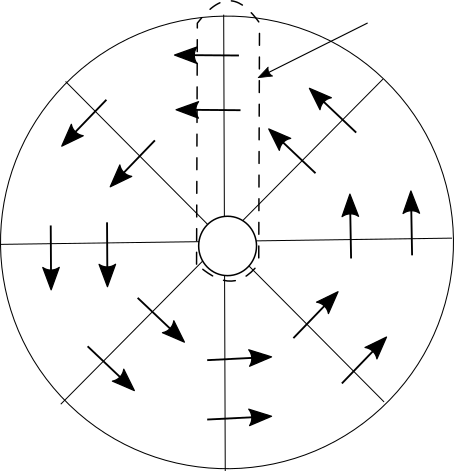}

    \put(40,120){$\Sigma_0\simeq \Sigma_1$}
      \put(115,60){$\Sigma_{\frac{3}{4}}$}
      \put(-15,60){$\Sigma_{\frac{1}{4}}$}
      \put(55,-8){$\Sigma_{\frac{1}{2}}$}
      \put(92,112){$\epsilon$-neighborhood}
      \put(53,53){B}

  \end{overpic}
  \vskip0.2cm
\caption{Using open book decomposition and the flow of $X_\alpha$ to establish almost global realization}
\end{figure}
\begin{proof}
Let $(B,\pi)$ be an open book decomposition adapted to $\xi$ and $\alpha$ a contact form for $\xi$ satisfying the condition of Definition~\ref{adapted}. After multiplying $\alpha$ by a constant, we can assume $Vol(g)=\frac{1}{\theta'} \alpha \wedge d\alpha=V$, for any compatible metric $g$ associated with $\alpha$.

Choose an arbitrary complex structure $J$ on $\xi$, inducing the compatible metric $g$. Parametrizing $\mathbb{S}^1 \simeq [0,1] / 0\sim 1$, consider $J|_{\Sigma_{0}\backslash B}$ to be initial condition for the PDE described in Lemma~\ref{pert} and since the interior of pages of $(B,\pi)$ are transverse to $X_\alpha$, we can solve such PDE (as in local realization theorem) and extend the solution of realization problem over $\Sigma_\tau \backslash B$ for $0< \tau <1$, i.e. $M\backslash \Sigma_{0}$. The achieved complex structure $J_{(\lambda,\eta )}$ on $M\backslash \Sigma_{0}$ yields a singular (measurable) compatible metric $g_\infty$, satisfying 1)-3) with $F:=\Sigma_{0}$ being the singular set. Also note that the volume form of $g_\infty$ is the same as $g$, since $F$ is measure zero. See Figure 3.

Now we can realize the measurable semi-metric $g_\infty$ as the limit of a $L^2$-Cauchy sequence of compatible metric, using Theorem~\ref{metchar} and Theorem~\ref{convchar} in the following way. For any fixed $\epsilon>0$, choose small enough $\delta >0$, such that $$Vol(E:= \bigcup\limits_{1-\delta \leq \tau \leq 1}  \left ( \Sigma_\tau \backslash B \right ) , g ) <\frac{\epsilon}{2}.$$ Now since $X_\alpha$ induces a product structure on $E$, we can use a smooth interpolation function $h_\delta :[0,1]\to [0,1]$ with $h_\delta(\tau)=0$ for $0 \leq \tau \leq 1-\delta $ and $h_\delta(1)=1$ and see that the the complex structure $$\tilde{J_\epsilon}|_{\Sigma_\tau}:= (1-h_\delta(\tau))J_{(\lambda,\eta)} + h_\delta(\tau) J$$ for $\tau \in \mathbb{S}^1$ can be extended over $\Sigma_0 \backslash B \simeq \Sigma_1 \backslash B$, yielding a singular compatible metric, which is singular on $B$ and has $Ricci(X_\alpha)=f(x)$ outside of a $\frac{\epsilon}{2}$-neighborhood of $\Sigma_0$. Similarly, with a smooth radial interpolation between $\tilde{J_\epsilon}$ and $J$ in a product $\frac{\epsilon}{2}$-neighborhood of $B$, we can define $J_\epsilon$ and consequently the compatible $g_\epsilon$ on all of $M$, such that $Ricci(X_\alpha)=f(x)$ outside of a $\epsilon$-neighborhood of $\Sigma_0$. We claim that repeating this procedure for $\epsilon_n :=\frac{\epsilon}{2^n}$ gives the sequence described in 4).

First, notice that for all the metrics above, we only perturbed the complex structure, leaving the volume form unchanged. Therefore by Proposition~\ref{cool}, $$d(g_{\epsilon_n}, g_{\epsilon_m})\leq 2C \sqrt{\frac{\epsilon}{2^{min\{ m,n \}}}}$$ and $g_{\epsilon_n}$ is a Cauchy sequence and moreover satisfies condition 1) of Theorem~\ref{convchar}. Now, note that for any $x\in M\backslash \Sigma_0$, there exists $N\in \mathbb{N}$ such that for $n\geq N$, $g_{\epsilon_n}=g_\infty$ and hence, we have condition 3) of Theorem~\ref{convchar}. That also means that $D_{ \{ g_{\epsilon_n}\} }$ is included in the measure zero set $X_{g_\infty}=F=\Sigma_0$. Therefore by Theorem~\ref{convchar},  $\{ g_{\epsilon_n}\}_{n\in\mathbb{N}}$ $L^2$-converges to $[g_\infty] \in \overline{\mathcal{M}_\xi}$.

\end{proof}


\Addresses


\begin{thebibliography}{00}

\bibitem{alexander} J. W. Alexander, {\em Note on Riemann spaces}, Bull. Amer. Math. Soc., 26 (1920), 370–372.
\bibitem{berger} Berger M., {\em Les vari'et'es Riemanniennes 1/4-pinc'ees}, Ann. Scuola Norm. Sup. Pisa 14, 161–170 (1960)
\bibitem{blsign} Blair, David. (2019). {\em On the Sign of the Curvature of a Contact Metric Manifold}. Mathematics. 7. 892. 10.3390/math7100892. 
\bibitem{bl} Blair, David E. {\em Riemannian geometry of contact and symplectic manifolds.} Springer Science \& Business Media, 2010.
\bibitem{blperr} Blair, David E., and D. Peronne. {\em Conformally Anosov flows in contact metric geometry.} Balkan Journal of Geometry and Its Applications 3.2 (1998): 33-46.
\bibitem{bw} Boothby, William M., and Hsieu-Chung Wang. {\em On contact manifolds.} Annals of Mathematics (1958): 721-734.
\bibitem{boy} Boyer, Charles P., Leonardo Macarini, and Otto van Koert. {\em Brieskorn manifolds, positive Sasakian geometry, and contact topology.} Forum Mathematicum. Vol. 28. No. 5. De Gruyter, 2016.
\bibitem{BC1} Clarke, Brian. {\em The Completion of the Manifold of Riemannian Metrics with Respect to its $ L^ 2$ Metric.} arXiv preprint arXiv:0904.0159 (2009).
\bibitem{BC2} Clarke, Brian. {\em The metric geometry of the manifold of Riemannian metrics over a closed manifold.} Calculus of Variations and Partial Differential Equations 39.3-4 (2010): 533-545.
\bibitem{elover} Eliashberg, Yakov. {\em Classification of overtwisted contact structures on 3-manifolds.} Inventiones mathematicae 98.3 (1989): 623-637.
\bibitem{elmart} Eliashberg, Yakov. {\em Contact 3-manifolds twenty years since J. Martinet's work.} Annales de l'institut Fourier. Vol. 42. No. 1-2. 1992.
\bibitem{confoliations} Eliashberg, Yakov, and William P. Thurston. {\em Confoliations.} Vol. 13. American Mathematical Soc., 1998.
\bibitem{etnyreintro} Etnyre, John B. {\em Introductory lectures on contact geometry.} arXiv preprint math/0111118 (2001).
\bibitem{OBD} Etnyre, John B. {\em Lectures on open book decompositions and contact.} Floer Homology, Gauge Theory, and Low-dimensional Topology: Proceedings of the Clay Mathematics Institute 2004 Summer School, Alfréd Rényi Institute of Mathematics, Budapest, Hungary, June 5-26, 2004. Vol. 5. American Mathematical Soc., 2006.
\bibitem{InvOp} Etnyre, John, and Burak Ozbagci. {\em Invariants of contact structures from open books.} Transactions of the American Mathematical Society 360.6 (2008): 3133-3151.
\bibitem{EKM1} Etnyre, John B., Rafal Komendarczyk, and Patrick Massot. {\em Tightness in contact metric 3-manifolds.} Inventiones mathematicae 188.3 (2012): 621-657.
\bibitem{EKM2} Etnyre, John, Rafal Komendarczyk, and Patrick Massot. {\em Quantitative Darboux theorems in contact geometry.} Transactions of the American Mathematical Society 368.11 (2016): 7845-7881.
\bibitem{JH} Ge, Jian, and Yang Huang. {\em 1/4-Pinched Contact Sphere Theorem."} arXiv preprint arXiv:1304.5224 (2013).
\bibitem{his} Geiges, Hansjörg. {\em A brief history of contact geometry and topology.} Expositiones Mathematicae 19.1 (2001): 25-53.
\bibitem{Geiges} Geiges, Hansjörg. {\em Contact geometry.} Handbook of differential geometry. Vol. 2. North-Holland, 2006. 315-382.
\bibitem{giroux} E. Giroux, {\em G'eom'etrie de contact: de la dimension trois vers les dimensions sup'erieures}, Proceedings of the International Congress of Mathematicians, Vol. II (Beijing, 2002), 405–414, Higher Ed. Press, Beijing, 2002.
\bibitem{hamilton} Hamilton, Richard S. {\em Three-manifolds with positive Ricci curvature}. J. Differential Geom. 17 (1982), no. 2, 255--306. doi:10.4310/jdg/1214436922. https://projecteuclid.org/euclid.jdg/1214436922
\bibitem{hoferdy} Hofer, Helmut, Markus Kriener. {\em Holomorphic curves in contact dynamics.} Proceedings of Symposia in Pure Mathematics. Vol. 65. AMERICAN MATHEMATICAL SOCIETY, 1999.
\bibitem{Hoz} Hozoori, Surena. {\em Dynamics and Topology of Conformally Anosov Contact 3-Manifolds. } arXiv preprint arXiv:1908.07990 (2019).
\bibitem{HT} Hutchings, Michael, and Clifford Henry Taubes. {\em The Weinstein conjecture for stable Hamiltonian structures.} Geometry \& Topology 13.2 (2009): 901-941.
\bibitem{kling} Klingenberg W., Uber Riemannsche Mannigfaltigkeiten mit positiver Krummung,
Comment. Math. Helv. 35, 47–54 (1961)
\bibitem{Kob} Kobayashi, Shoshichi. {\em Principal fibre bundles with the 1-dimensional toroidal group.} Tohoku Mathematical Journal, Second Series 8.1 (1956): 29-45.
\bibitem{kroug} Krouglov, Vladimir. {\em A note on the conjecture of Blair in contact Riemannian Geometry.} Tohoku Mathematical Journal, Second Series 64.4 (2012): 561-567.
\bibitem{lokhamp} Lohkamp, Joachim. {\em Metrics of Negative Ricci Curvature.} Annals of Mathematics, vol. 140, no. 3, 1994, pp. 655–683. JSTOR, www.jstor.org/stable/2118620.
\bibitem{Massot} Massot, Patrick. {\em Geodesible contact structures on 3–manifolds.} Geometry \& Topology 12.3 (2008): 1729-1776.
\bibitem{Mitsumatsu} Mitsumatsu, Yoshihiko. {\em Anosov flows and non-Stein symplectic manifolds.} Annales de l'institut Fourier. Vol. 45. No. 5. 1995.
\bibitem{Nied} Niederkrüger, Klaus, and Federica Pasquotto. {\em Resolution of symplectic cyclic orbifold singularities.} arXiv preprint arXiv:0707.4141 (2007).
\bibitem{perr} Perrone, Domenico. {\em Torsion and conformally Anosov flows in contact Riemannian geometry.} Journal of Geometry 83.1-2 (2005): 164-174.
\bibitem{Ruk} Rukimbira, Phillippe. {\em Chern-Hamilton Conjecture and K-contactness.} Hous. Jour. of Math. 21, No 4, (1995), 709-718.
\bibitem{Sakai} Sakai, Takashi. {\em Riemannian Geometry}. Translations of Mathematical Monographs, vol. 149. American Mathematical Society, Providence, RI (1996): 262-272.
\bibitem{AlReg} Thomas, C. B. {\em Almost regular contact manifolds.} Journal of Differential Geometry 11.4 (1976): 521-533.
\bibitem{thurwilk} W. P. Thurston, H. Winkelnkemper, {\em On the existence of contact forms}, Proc. Amer. Math. Soc. 52 (1975), 345–347.


\end{thebibliography}
\end{document}